\documentclass[conference]{IEEEtran}
\IEEEoverridecommandlockouts
\usepackage{cite}
\usepackage{mydef}
\usepackage{amsmath,amssymb,amsfonts}
\usepackage{algorithmic}
\usepackage{graphicx}
\usepackage{textcomp}
\usepackage{xcolor}
\def\BibTeX{{\rm B\kern-.05em{\sc i\kern-.025em b}\kern-.08em
    T\kern-.1667em\lower.7ex\hbox{E}\kern-.125emX}}
\begin{document}

\title{A Multi-Agent Primal-Dual Strategy for Composite Optimization over Distributed Features
}

\author{\IEEEauthorblockN{Sulaiman A. Alghunaim}
\IEEEauthorblockA{\textit{Department of Electrical Engineering} \\
\textit{Kuwait University}\\
\textit{Kuwait} \\
salghunaim@ucla.edu}
\and
\IEEEauthorblockN{Ming Yan}
\IEEEauthorblockA{\textit{Departments of CMSE \& Math} \\
\textit{Michigan State University}\\
\textit{East Lansing, MI, USA} \\
myan@msu.edu}
\and 
\IEEEauthorblockN{ Ali H. Sayed}
\IEEEauthorblockA{\textit{School of Engineering} \\
\textit{Ecole Polytechnique Federale de Lausanne} \\
\textit{Switzerland}\\
ali.sayed@epfl.ch}
}
\maketitle
\begin{abstract}
This work studies multi-agent sharing optimization problems with the objective function being the sum of smooth local functions plus a convex (possibly non-smooth) function coupling all agents. This scenario arises in many machine learning and engineering applications, such as regression over distributed features and resource allocation. We reformulate this problem into an equivalent saddle-point problem, which is amenable to decentralized solutions. We then propose a proximal primal-dual algorithm and establish its linear convergence to the optimal solution when the local functions are strongly-convex. To our knowledge, this is the first linearly convergent decentralized algorithm for multi-agent sharing problems with a general convex (possibly non-smooth) coupling function.  
\end{abstract}


\section{Introduction}
 We consider $K$ agents connected through a graph. Each agent can only send and receive information from its immediate neighbors. Its goal is to find its corresponding solution, denoted by $w^\star_k \in \real^{Q_k}$, of the following coupled multi-agent optimization problem:
\eq{ 
 \min_{w_1,\cdots,w_K} \quad   \sum_{k=1}^K J_k(w_k)+g\Big(\sum_{k=1}^K B_k w_k \Big), \label{min-generally-coupled}
} 
where the smooth function $J_k:\real^{Q_k} \rightarrow \real$ and the matrix $B_k \in \real^{E \times Q_k}$ are known by agent $k$ only, and $g:\real^{E} \rightarrow \real \cup \{+ \infty\}$ is a convex possibly non-smooth function known by all agents. Problem \eqref{min-generally-coupled} is the sharing formulation, where the agents own different variables but are coupled through the function $g$. Problems of the form \eqref{min-generally-coupled} appear in many machine learning applications, such as regression over distributed features~\cite{boyd2011admm,sundhar2012new}, dictionary learning over distributed models~\cite{chen2014dictionary}, and clustering in graphs~\cite{hallac2015network}. They also appear in engineering applications, including smart grid control~\cite{chang2014distributed} and network utility maximization~\cite{palomar2006tutorial}. For a general convex function $g$,  centralized algorithms for \eqref{min-generally-coupled} have been shown to achieve global linear convergence if the matrix $[B_1, \cdots,B_K]$  has full row rank and $\sum\limits_{k=1}^K J_k(w_k)$ is strongly-convex~\cite{chen2013aprimal,dhingra2018proximal,deng2016global}.  On the other hand, {\em decentralized} algorithms have only been shown to converge linearly under stricter conditions than centralized ones.  In this work, we aim toward closing the gap in linear convergence between centralized and decentralized algorithms for problem~\eqref{min-generally-coupled}.

{\bf Literature Review.} Sharing problems have been studied in different fields and date back to studies in economics~\cite{walras1896} --see the discussion in~\cite{boyd2011admm}. The earliest center-free algorithm to solve such problems dates back to~\cite{ho1980class}. Decentralized solutions for the sharing formulation~\eqref{min-generally-coupled} have only been shown to achieve global linear convergence for special cases and under stricter conditions compared to the centralized ones as we now explain. The works~\cite{necoara2013random,doan2017distributed,nedic2018improved,xu2018dual} establish linear convergence for resource allocation formulations, where $g$ is an indicator function of zero (i.e., $g(x)=0$ if $x= 0$ and $\infty$ otherwise) and $B_k=I$, for smooth and strongly-convex local costs. The work~\cite{yang2019distributed} also establishes linear convergence for resource allocation problems in the presence of simple local constraints (i.e., $\underline{w} \leq w_k\leq \overline{w}$), but under stronger assumptions on the costs such as twice differentiability of $J_k$ and knowledge of the conjugate function of $J_k$. The works~\cite{chang2015multi,alghunaim2018proximal} establish linear convergence for special cases with $g$ being an indicator function of zero; moreover, each $B_k$ is required to have full row rank in~\cite{chang2015multi} or satisfy a certain rank condition in~\cite{alghunaim2018proximal}. The works \cite{ying2018supervised,he2018cola} establish linear convergence for strongly convex objectives and a smooth coupling function $g$.

Note that problem \eqref{min-generally-coupled} recovers the consensus problem~\cite{tsitsiklis1986distributed,nedic2009distributed,sayed2014nowbook} if we choose $g$ such that it returns zero when $w_1=\cdots=w_K$ and $\infty$ otherwise. In this case, the matrix $B=[B_1, \cdots,B_K]$ is sparse and encodes the communication graph between agents. The works~\cite{sun2019convergence,alghunaim2019linearly,alghunaim2019decentralized} studied linear convergence of consensus problems in the presence of a common non-smooth term, which is not applicable for the sharing problem. Different from the consensus problem, the matrix $B$ in the sharing problem is not necessarily sparse, and $B_k$ is a private matrix known by agent $k$ only.  Thus, solution methods for these two problems are different \cite{boyd2011admm}.   
To the best of our knowledge, establishing linearly convergent algorithms for the {\em sharing} problem~\eqref{min-generally-coupled} with a general convex $g$ are still missing.

{\bf Contribution.}  As mentioned before, there is a theoretical gap between centralized and decentralized algorithms for the sharing problem~\eqref{min-generally-coupled}. In this work, we propose a decentralized algorithm for~\eqref{min-generally-coupled} and establish its linear convergence to the global solution. The derivation of our algorithm is based on reformulating \eqref{min-generally-coupled} into an equivalent problem that is amenable to decentralized solutions. This technique motivates the derivation of many other decentralized algorithms.  
 
{\bf Notation.}  We let $\|x\|^2_{D}=x\tran Dx$ for a square matrix $D$. The symbol $I_S$ denotes the identity matrix of size $S$, and $S$ is removed when there is no confusion. The symbol $\one_N$ denotes the $N \times 1$ vector with all entries being one. The Kronecker product of two matrices is denoted by $A\otimes B$.  We use ${\rm col}\{x_k\}_{k=1}^{K}$ to denote a column vector formed by stacking $x_1,\cdots, x_K$ on top of each other and $\text{blkdiag}\{X_k\}_{k=1}^{K}$ to denote a block diagonal matrix consisting of diagonal blocks $\{X_k\}$.  The subdifferential $\partial f(x)$ of a function $f$ at $x$ is the set of all subgradients. 
The proximal operator of a function $f(x)$ with step-size $\mu$ is  $
{\rm prox}_{\mu f}(x)= \argmin\limits_u \ f(u)+{1 \over 2 \mu} \|x-u\|^2$.  
The conjugate of a function $f$ is defined as $f^*(v)=\sup\limits_{x} \ v\tran x -f(x)$. A differentiable function $f$ is $\delta$-smooth and $\nu$-strongly-convex if $\|\grad f(x)-\grad f(y)\| \leq \delta \|x-y\|$ and $(x-y)\tran \big(\grad f(x)-\grad f(y)\big) \geq \nu \|x-y\|^2$, respectively, for any $x$ and $y$. 

\section{Saddle-Point Reformulation}
In this section, we provide the main assumption on the objective and explain how \eqref{min-generally-coupled} is reformulated into an equivalent saddle-point problem.  We introduce the quantities:
\begin{subequations}
\eq{
\sw &\define {\rm col}\{w_1, \cdots,w_K\} \in \real^Q, ~ Q \textstyle \define \sum_{k=1}^K Q_k, \\
\cJ(\sw) &\define \textstyle \sum_{k=1}^K J_k(w_k), \\
B &\define  \begin{bmatrix}
B_1 & \cdots & B_K
\end{bmatrix} \in \real^{E \times Q}.  \label{B_cen}
}
\end{subequations}
Then, problem \eqref{min-generally-coupled} can be rewritten as: 
\eq{ 
\min_{\ssw} \quad   \cJ(\sw)+g\left(B \sw \right). 
\label{primal_reformulated}
}
Throughout this work, the following assumption holds.
\begin{assumption} \label{assump:cost}
{ {\bf (Objective Functions)} Problem \eqref{primal_reformulated} has a solution $\sw^\star$, and the function $\cJ:\real^{Q} \rightarrow \real$ is $\delta$-smooth 
and $\nu$-strongly-convex with $0<\nu\leq\delta$. Moreover, the function $g:\real^{E} \rightarrow \real \cup \{+ \infty\}$ is proper lower semi-continuous and convex, and there exists $\sw \in \real^Q$ such that $B\sw$ belongs to the relative interior domain of $g$.   }
\end{assumption}
Under Assumption \ref{assump:cost}, strong duality holds \cite[Corollary 31.2.1]{rockafellar1970convex}, and problem \eqref{primal_reformulated} is equivalent to the following saddle-point problem \cite[Proposition 19.18]{bauschke2011convex}:
\eq{ \boxed{
\min_{\ssw} \max_{y} \quad   \cJ(\sw)+ y\tran B \sw  -g^*(y)} \label{saddle_point_cen}
} 
where $y$ is the dual variable and  $g^*$ is the conjugate function of $g$. A primal-dual pair $(\sw^\star, y^\star)$ is  optimal if, and only if, it satisfies the optimality conditions \cite[Proposition 19.18]{bauschke2011convex}:
\begin{subequations} \label{KKT}
\eq{
-B\tran y^\star & = \grad \cJ(\sw^\star)  \label{kkt1}, \\
B \sw^\star & \in   \partial g^*(y^\star) \label{kkt2}.
}
\end{subequations}
Directly solving \eqref{saddle_point_cen} does not result in a decentralized algorithm. This is because the dual variable $y$ couples all agents, and these algorithms would require a centralized unit to compute the dual update. Therefore, further reformulations are needed to derive a decentralized solution.  
\section{Decentralized Reformulation}
In this section, we reformulate \eqref{saddle_point_cen} into another equivalent saddle-point problem that can be solved in a decentralized manner.  
Since the dual variable $y$ couples all agents, we introduce local copies of $y$ at all agents. Let $y_k$ denote a local copy of $y$ at agent $k$. 
In addition, we introduce the following  network quantities:
\begin{subequations}
\eq{
\sy&\define {\rm col}\{y_k\}_{k=1}^K  \in \real^{EK},\quad\cG^*(\ssy)\define{1 \over K}  \sum_{k=1}^K g^*(y_k), \nonumber \\
\cB_d&\define{\rm blkdiag}\{B_k\}_{k=1}^K\in \real^{EK\times Q},\nonumber
} 
\end{subequations}
and the symmetric matrix $\cL \in \real^{EK \times EK}$ such that:
\eq{
\cL \sy =0 \iff y_1=\cdots =y_K. \label{consensus_matrix}
}
 Consider the saddle-point problem:
\eq{
\boxed{ \min_{\ssw,\ssx} \max_{\ssy}  \quad   \cJ(\sw)+ \ssy\tran \cB_d \sw+\ssy\tran \cL \ssx -\cG^*(\sy) } \label{saddle_point_dec}
} 
with an optimal solution $(\sw^\star,\ssx^\star,\ssy^\star)$ satisfying \cite{bauschke2011convex}:
\begin{subequations} \label{kkt_dec}
\eq{
-\cB_d\tran \sy^\star &=\grad\cJ(\sw^\star),  \label{kkt1_dec} \\
\cL \sy^\star&=   0, \label{kkt2_dec} \\ 
 \cB_d \sw^\star+\cL \ssx^\star & \in    \partial \cG^*(\sy^\star). \label{kkt3_dec} 
}
\end{subequations}
 Problem \eqref{saddle_point_dec} can be solved in a decentralized manner because the matrix $\cL$ encodes the network sparsity structure and the matrix $\cB_d$ is block diagonal. We now show that problems \eqref{saddle_point_dec} and \eqref{saddle_point_cen} are equivalent. 

\begin{lemma} {\bf (Saddle-Point)} \label{lemma_kkt_dec}
If $(\sw^\star,\ssx^\star,\ssy^\star)$ satisfies the optimality condition \eqref{kkt_dec}, then it holds that $\sy^\star=\one_K \otimes y^\star$ with $(\sw^\star, y^\star)$ satisfying the optimality condition \eqref{KKT}.
\end{lemma}
\begin{proof}
From equations \eqref{consensus_matrix} and \eqref{kkt2_dec}, we have that $\sy^\star= \one_K \otimes y^\star$ for some $y^\star$. Thus, equation \eqref{kkt1_dec} can be written as:
\eq{
-\cB_d\tran \sy^\star=-B\tran y^\star =\grad\cJ(\sw^\star). \label{lklk1}
}
Multiplying \eqref{kkt3_dec} by $\one_K\tran \otimes I_E$ and using $\sy^\star= \one_K \otimes y^\star$ gives:
\eq{
&(\one_K\tran \otimes I_E)\cB_d \sw^\star+(\one_K\tran \otimes I_E)\cL \ssx^\star & \in   (\one_K\tran \otimes I_E) \partial \cG^*(\sy^\star) \nnb
&\Longrightarrow  B \sw^\star \in \partial g^*(y^\star), \label{lklk2}
}
where we used the fact that $(\one_K\tran \otimes I_E)\cL=0$, which holds from \eqref{consensus_matrix}. From equations \eqref{lklk1} and \eqref{lklk2}, we see that $(\sw^\star, y^\star)$ satisfies the optimality condition \eqref{KKT}. 
\end{proof}

\section{Decentralized Strategy}
In this section, we propose an algorithm to solve \eqref{saddle_point_dec} and show how to implement it in a decentralized manner.
\subsection{General Algorithm}
Let $\sw_{-1},\sy_{-1}$ be any values and $\ssx_{-1}=0$. The iteration is:
\begin{subequations} \label{alg_dec}
\eq{
\sw_i&=\sw_{i-1}-\mu_w \grad \cJ(\sw_{i-1})-\mu_w \cB\tran_{d} \sy_{i-1}, \label{W_update_alg} \\
\ssz_{i}&= \sy_{i-1}+\mu_y \cB_{d}  \sw_i  + \cL \ssx_{i-1}, \label{Z_update_alg} \\
\ssx_{i} &= \ssx_{i-1}- \cL \ssz_{i}, \label{X_update_alg} \\
\sy_i&= {\rm prox}_{\mu_y \cG^*}(\bar{\cA} \ssz_{i}), \label{Y_update_alg}
}
\end{subequations}
where $\bar{\cA}=\bar{A} \otimes I_E$ with $\bar{A} \in \real^{K \times K}$ satisfying Assumption \ref{assump:matrix}.  
\begin{assumption} \label{assump:matrix}
{ {\bf (Combination Matrices)}  We assume that $\bar{A}$ is a primitive symmetric doubly-stochastic matrix. Moreover, we assume that the matrix $\cL$ satisfies condition \eqref{consensus_matrix} and 
\begin{align}      0<I-\cL^2, 
\quad   \bar{\cA}^2 \leq I-\cL^2. \label{AleqL_assump}
\end{align}}
\end{assumption}
The condition on $\bar{A}$ can be easily satisfied for any undirected connected network -- see~\cite{sayed2014nowbook}. Note that the eigenvalues of $\bar{\cA}$ belong to $(-1,1]$. Given $\bar{\cA}$, there are many choices for $\cL$. For instance, we can let $\cL^2=I-\bar{\cA}^2$ and check whether $I-\cL^2$ is positive definite or not. If it is positive definite, then the assumption is satisfied. If not, we can let $\cL^2=c(I-\bar{\cA}^2)$ for any $c\in(0,1)$.
Although many choices for $\cL$ and $\bar\cA$ are possible, due to space considerations, we only focus on one choice in this work.

We first construct a primitive symmetric doubly-stochastic matrix $A=[a_{sk}]\in \real^{K\times K}$ such that $a_{sk}=0$ if two agents $s$ and $k$ are not connected through an edge. Then, we let $\bar A = 0.5(I+A)$, which is also a primitive symmetric doubly-stochastic matrix. In this case, the eigenvalues of $\bar\cA$ belong to $(0,1]$, and we can let $\cL^2=I-\bar\cA$, which satisfies Assumption \ref{assump:matrix}. We now show how to implement \eqref{alg_dec} using these choices.


\subsection{Proximal Exact Dual Diffusion (PED$^2$)} \label{sec_dual_prox_diff}
 From \eqref{Z_update_alg}--\eqref{X_update_alg}, it holds that, for all $i\geq 1$,
\eq{
\hspace{-1mm }  \ssz_{i}& \hspace{-0.25mm }  = \hspace{-0.25mm }  (I-\cL^2) \ssz_{i-1}+ \sy_{i-1}-\sy_{i-2}+\mu_y \cB_{d}  (\sw_i -\sw_{i-1}). 
}
It eliminates $\ssx_i$, and we can rewrite \eqref{alg_dec} as
\eq{
\hspace{-1mm } \sw_i&=\sw_{i-1}-\mu_w \grad \cJ(\sw_{i-1})-\mu_w \cB\tran_{d} \sy_{i-1},  \nonumber\\
\hspace{-1mm } \ssz_{i}&  =(I-\cL^2) \ssz_{i-1}+ \sy_{i-1}-\sy_{i-2}+\mu_y \cB_{d}  (\sw_i -\sw_{i-1}),\nonumber\\
\hspace{-1mm } \phi_i &=  \bar{\cA} \ssz_{i},   \nonumber\\
\hspace{-1mm } \sy_i&= {\rm prox}_{\mu_y \cG^*}(\phi_i ).\nonumber
}
With our choice of $\bar\cA$ and $\cL$, the $k$-th block in the vectors can be updated by agent $k$. Let $w_{k,-1},~y_{k,-1}$ be any values and $\phi_{k,-1}=\psi_{k,-1}$. For each agent $k$, repeat for $i \geq 0$:
\begin{subequations} \label{exac_proxdiff_dual}
\eq{
w_{k,i}&=w_{k,i-1}-\mu_w \grad J_k(w_{k,i-1})-\mu_w B_k\tran y_{k,i-1}, \\
\psi_{k,i}&= y_{k,i-1}+\mu_y B_{k}  w_{k,i}, \\
z_{k,i} &= \phi_{k,i-1}+\psi_{k,i}-\psi_{k,i-1}, \\
\phi_{k,i}&= \sum_{s \in \cN_k}\bar{a}_{sk} z_{s,i}, \\
y_{k,i}&= {\rm prox}_{\mu_y/K g^*}(\phi_{k,i}). 
}
\end{subequations}

\section{Linear Convergence Result}
In this section, we establish the linear convergence of~\eqref{alg_dec}.  We first show that the fixed-point of \eqref{alg_dec} is optimal.
\begin{lemma} {\bf (Fixed Point)} \label{lemma_optimiality}
A  fixed-point $(\sw^o,\ssx^o,\ssy^o, \ssz^o)$ of \eqref{alg_dec} exists, i.e., 
\begin{subequations} \label{optimality_conditions}
\eq{
0&=\grad\cJ(\sw^o) + \cB_d\tran \sy^o, \label{optimality1} \\
\ssz^o&= \sy^o+\mu_y \cB_{d}  \sw^o  + \cL \ssx^o, \label{optimality2} \\
 0&=\cL \ssz^o,  \label{optimality3} \\
 \sy^o&= {\rm prox}_{\mu_y \cG^*}(\bar{\cA} \ssz^o). \label{optimality4} 
}
\end{subequations}
Also, for any fixed-point $(\sw^o,\ssx^o,\ssy^o, \ssz^o)$, it holds that $\sy^o=\one_K \otimes y^o$ with $(\sw^o, y^o)$ satisfying the optimality condition \eqref{KKT}.
\end{lemma}
\begin{proof}
Given an optimal solution $(\sw^\star,y^\star)$ of~\eqref{saddle_point_cen} that satisfies \eqref{KKT}, we let $\sw^o \define \sw^\star$ and $\sy^o \define \sy^\star=\one_K\otimes y^\star$. Then, \eqref{optimality1} holds because of \eqref{kkt1_dec}. We define 
\eq{
\ssz^o \define \one_K\otimes (y^\star+ {\mu_y\over K}B\sw^\star)\define \one_K\otimes z^o,\label{lem111}
}
which satisfies condition \eqref{optimality3}.  
Because of the construction of $\bar{\cA}$ and~\eqref{optimality3}, we have $\bar{\cA} \ssz^o= \ssz^o$, and  equation~\eqref{optimality4} is equivalent to $\ssz^o-\sy^o \in \mu_y \partial \cG^*(\sy^o)$.  Thus, using the definition of $\cG^*$,  equation~\eqref{optimality4}  holds from \eqref{kkt2} and \eqref{lem111}. Finally, we construct $\ssx^o$ such that \eqref{optimality2} holds. To see this, note that
\eq{
& (\one_K^\top\otimes I_E)(\ssz^o- \sy^o-\mu_y \cB_{d}  \sw^o ) \nnb
 = & K (z^o-y^o)-\mu_y B \sw^o \overset{\eqref{lem111}}{=}0
}
The above equation implies that $\ssz^o- \sy^o-\mu_y \cB_{d}  \sw^o $ belongs to the range space of $\cL$ (null space of $(\one_K^\top\otimes I_E)$). Thus, there exists $\ssx^o$ such that \eqref{optimality2} holds.


Now, if $(\sw^o,\ssx^o,\ssy^o, \ssz^o)$ is a fixed-point of \eqref{alg_dec}, then we have $\ssz^o=\one_K\otimes z^o$ because of~\eqref{optimality3}. Therefore, $\bar{\cA} \ssz^o= \ssz^o$, and~\eqref{optimality4} shows that $\ssy^o=\one_K \otimes y^o$. Combining~\eqref{optimality2} and~\eqref{optimality4}, we have $\cB_d\sw^o+{1\over \mu_y}\cL\ssx^o\in\partial \cG^*(\ssy^o)$. Therefore, $(\sw^o,y^o)$ satisfies~\eqref{KKT}, which follows from Lemma \ref{lemma_kkt_dec}.
\end{proof}
Note from \eqref{X_update_alg} that if $\ssx_{-1}=0$, then $\ssx_{1} = - \cL \ssz_{1}$ belongs to the range space of $\cL$. Consequently, $\{\ssx_i\}_{i \geq 0}$ will always remain in  the range space of $\cL$. By following similar arguments to~\cite[Lemma 2]{alghunaim2019linear}, we can always assume that $(\sw^o,\ssx^o,\ssy^o,\ssz^o)$ is a fixed-point with $\ssx^o$ in the range space of $\cL$ because adding a vector in the null space of $\cL$ to $\ssx^o$ does not change the optimality condition. To analyze the algorithm \eqref{alg_dec}, we consider the error quantities:
\begin{subequations}
\eq{
\tsw_i&=\sw_i-\sw^o, \ &\tsy_i&=\ssy_i-\ssy^o, \nonumber\\
\tsz_i&=\ssz_i-\ssz^o, \ &\tsx_{i}&=\ssx_i-{\ssx}^o. \nonumber
}
\end{subequations}
From \eqref{alg_dec} and \eqref{optimality_conditions}, the error quantities evolve as:
\begin{subequations}
\eq{
\tsw_i&=\tsw_{i-1}-\mu_w (\grad \cJ(\sw_{i-1})-\grad \cJ(\sw^o))-\mu_w \cB\tran_{d} \tsy_{i-1}, \label{err_w}\\
\tsz_{i}&= \tsy_{i-1}+\mu_y \cB_{d}  \tsw_i  + \cL \tsx_{i-1}, \label{err_z} \\
\tsx_{i} &= \tsx_{i-1}- \cL \tsz_{i}, \label{err_x} \\
\tsy_i&= {\rm prox}_{\mu_y \cG^*}(\bar{\cA} \ssz_{i})-{\rm prox}_{\mu_y \cG^*}(\bar{\cA} \ssz^o). \label{err_y}
}
\end{subequations}
To state our main result, we note that condition \eqref{AleqL_assump} implies that  $0 \leq \cL^2 < I$. Therefore, $0<\underline{\sigma}(\cL)<1$, where  $\underline{\sigma}(\cL)$ denotes the smallest {\em non-zero} singular value of $\cL$. Let $\sigma_{\max}(\cL)$ denote the largest singular value of $\cL$. The following result establishes the linear convergence of the algorithm \eqref{alg_dec}.
\begin{theorem}\label{thm_lin_con}{\bf (Linear Convergence)} 
Let Assumptions \ref{assump:cost} and \ref{assump:matrix} hold. Assume that each $B_k$ has full row rank. If the step-sizes $\mu_w$ and $\mu_y$ are strictly positive and satisfy 
\eq{
\mu_w \leq \tfrac{2}{ \delta +\nu}, \quad \mu_y < \tfrac{2 \delta \nu}{(\delta+\nu)\sigma_{\max}^2(\cB_d)}, \label{step_size_c}
}
then it holds that $
\|\tsw_i\|^2 \leq \gamma^i
 C_o$ for all $i \geq 0$
and some $C_o\geq0$ where
\eq{
\gamma=\max & \bigg\{ \frac{1- {2\mu_w \delta \nu \over \delta+\nu}}{ 1-\mu_y \mu_w \sigma_{\max}^2(\cB_d)},1-\mu_w \mu_y\lambda_{\min}(\cB_d\cB_d\tran), \nnb
& \quad1-\underline{\sigma}^2(\cL) \bigg\}<1,
}
with $\lambda_{\min}(\cB_d\cB_d\tran)$ being the smallest eigenvalue of $\cB_d\cB_d\tran$.
\end{theorem}
\begin{proof}
See Appendix \ref{app-thmproof}.
\end{proof}
This theorem shows that the proposed algorithm has linear convergence for non-smooth $g$ if $B_k$ has full row rank. Centralized algorithms can achieve linear convergence when $B=[B_1 \ \cdots \ B_k]$ has full row rank. We leave it to future work to verify whether decentralized algorithms can also achieve linear convergence under the same condition. 
\begin{remark}[\sc Semi-strongly-convex] {\rm
Since \eqref{saddle_point_cen} has the same form as \eqref{saddle_point_dec}, it can be solved with existing algorithms. Thus, one can utilize existing algorithms to derive other decentralized  solutions for problem \eqref{saddle_point_dec}. However, existing linear convergence results \cite{chen2013aprimal,dhingra2018proximal,deng2016global} require the saddle-point to be strongly-convex with respect to (w.r.t.) the primal-variable. Problem \eqref{saddle_point_dec} is only strongly-convex w.r.t. the primal block vector $\sw$ but not strongly-convex w.r.t. to the whole variable ${\rm col}\{\sw,\ssx\}$. Therefore, the linear convergence results from \cite{chen2013aprimal,dhingra2018proximal,deng2016global} are not applicable in our setup. } 
\end{remark} 
\section{Numerical Simulation}
In this section, we apply algorithm \eqref{exac_proxdiff_dual}  to solve the following problem:
\eq{
\min_{w_1,\cdots, w_K} \quad  \frac{1}{2}  \sum_{k=1}^K  w_k\tran R_k w_k + r_k\tran w_k, \ \ s.t. \ \sum_{k=1}^K w_k \leq b.
}	  
This problem fits into \eqref{min-generally-coupled} with $B_k=I$ and $g(x)=0$ if $x \leq  b$ and $\infty$ otherwise. We randomly generate positive-definite matrices $R_k \in \real^{10 \times 10}$ and vectors $r_k \in \real^{10}$. The entries of the vector $b$ are uniformly chosen between $(0,1)$. The combination matrix $A$ is generated using the Metropolis rule \cite{sayed2014nowbook}. We consider a randomly generated network with $K=20$ agents shown on the left side of Fig.~\ref{fig:sim}. The simulation result is shown on the right side of Fig.~\ref{fig:sim}, where the linearized prox-ascent algorithm is:
\begin{subequations} \label{alg_prox_ascent}
\eq{
\sw_i&=\sw_{i-1}-\mu_w \grad \cJ(\sw_{i-1})-\mu_w B\tran \lambda_{i-1}, \\
\lambda_i&= {\rm prox}_{\mu_y g^*}(\lambda_{i-1}+\mu_y B  \sw_i).
}
The plot clearly shows PED$^2$~\eqref{exac_proxdiff_dual} converges linearly in this setup, and it is slightly slower than the centralized algorithm \eqref{alg_prox_ascent}.
\end{subequations}
\vspace{-5mm}
\begin{figure}[H]
\centering
	\includegraphics[width=\linewidth]{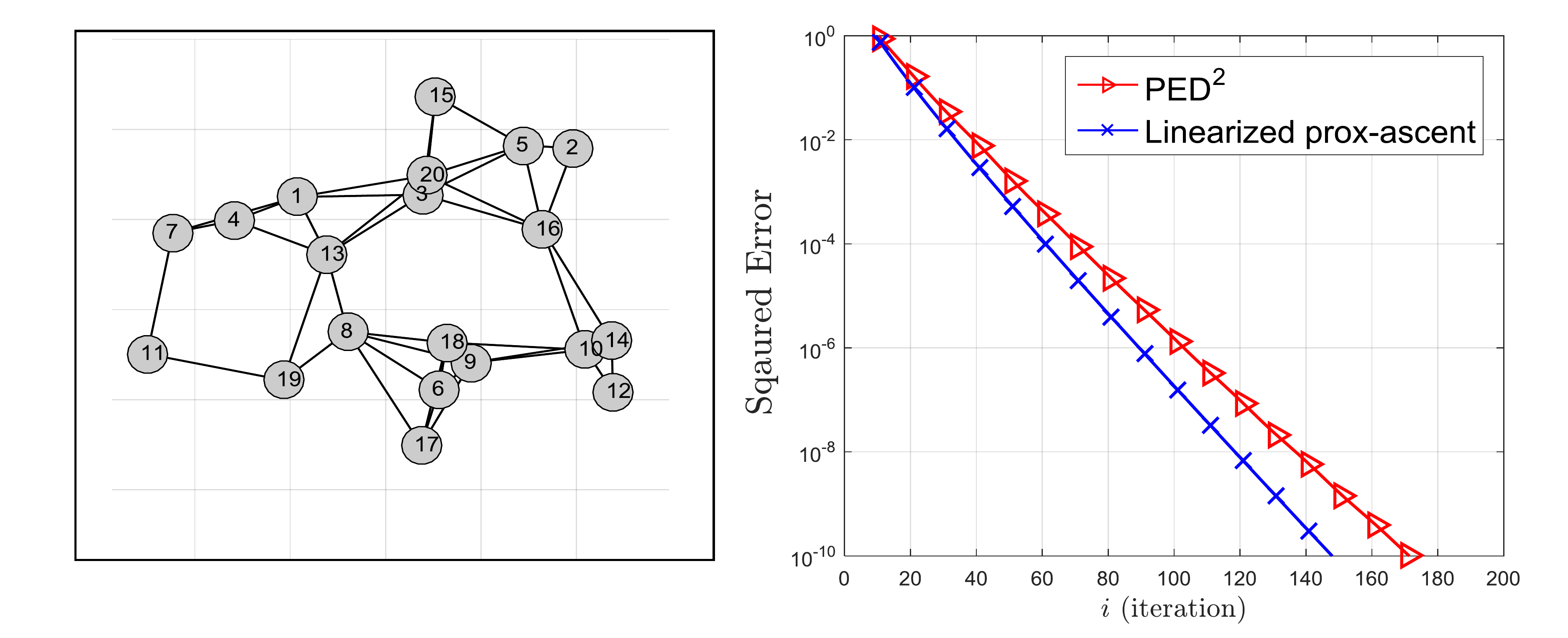}
\caption{ \small  The network topology used in the simulation and squared error $\|\sw_i-\sw^\star\|^2$ evolution of the proposed PED$^2$ \eqref{exac_proxdiff_dual} and the centralized algorithm \eqref{alg_prox_ascent} with $\mu_w=0.03$ and $\mu_y=2$.} \label{fig:sim}
\end{figure}

\section{Conclusion}
We studied the sharing problem \eqref{min-generally-coupled}, where agents are coupled through a convex possibly non-smooth composite function. To solve \eqref{min-generally-coupled} in a decentralized manner, we  reformulated it into an equivalent saddle-point problem. We then proposed a proximal decentralized algorithm and established its linear convergence. To our knowledge, this is the first {\em decentralized} linear convergence result for the multi-agent sharing problem \eqref{min-generally-coupled} with a general non-smooth coupling function.

\bibliographystyle{IEEEtran}
\bibliography{myref_review}

\appendices
\section{Proof of Theorem \ref{thm_lin_con}} \label{app-thmproof}
The following lemma establishes a useful equality, whose proof is omitted due to space limitations.
\begin{lemma} {\bf (Equality)} \label{lemma_equality}
Assume that the step-sizes $\mu_w$ and $\mu_y$ are strictly positive. The iterates of \eqref{alg_dec} satisfy:
\eq{
&\|\tsw_i\|^2_{I-\mu_y \mu_w \cB_d\tran \cB_d}+c_y\|\tsz_{i} \|_{I-\cL^2}^2+c_y\|\tsx_{i}\|^2 \nonumber \\
= &  \|\tsw_{i-1}-\mu_w (\grad \cJ(\sw_{i-1})-\grad \cJ(\sw^o))\|^2 \nnb
&  +c_y\|\tsy_{i-1}\|_{I-\mu_w \mu_y \cB_d\cB_d\tran}^2+ c_y\|\tsx_{i-1}\|_{I-\cL^2}^2, \label{equal_lemma}
}
where $c_y=\mu_w / \mu_y$. \qed
\end{lemma}
It can be verified that
\eq{
&\|\tsw_{i-1}-\mu_w (\grad \cJ(\sw_{i-1})-\grad \cJ(\sw^o))\|^2 \nnb
\leq& \big(1- \tfrac{2\mu_w \delta \nu}{ \delta+\nu}\big) \|\tsw_{i-1}\|^2 \nnb
\leq& \left({1- {2\mu_w \delta \nu \over \delta+\nu} \over 1-\mu_y \mu_w \sigma_{\max}^2(\cB_d)}\right) \|\tsw_{i-1}\|_{I-\mu_w \mu_y \cB_d\tran \cB_d}^2. \label{bound_grad_W}
}
where the first inequality holds under Assumption \ref{assump:cost} for  $\mu_w \leq {2 \over \delta+\nu}$ -- see \cite{nesterov2013introductory}. The second inequality holds if $I-\mu_y \mu_w \cB_d\tran \cB_d$ is positive definite.
Let 
\eq{
\gamma_1 \define {1- {2\mu_w \delta \nu \over \delta+\nu} \over 1-\mu_y \mu_w \sigma_{\max}^2(\cB_d)}.
}
Then, we have $\gamma_1<1$ if $\mu_y < {2 \delta \nu \over (\delta+\nu)\sigma_{\max}^2(\cB_d)}$. 
 Since the proximal mapping is nonexpansive, it holds from \eqref{err_y} that:
\eq{
\|\tsy_i\|^2&= \|{\rm prox}_{\mu_y \cG^*}(\bar{\cA} \ssz_{i})-{\rm prox}_{\mu_y \cG^*}(\bar{\cA} \ssz^o)\|^2 \nnb
&\leq \|\bar{\cA} \tsz_i\|^2= \| \tsz_i\|_{\bar{\cA}^2}^2 \leq  \|\tsz_i\|_{I-\cL^2}^2, \label{bound_Y_Z}
}
where the last inequality holds due to Assumption~\ref{assump:matrix}. Since each $B_k$ has full row rank, it holds that
\eq{
0< \lambda_{\min}(\cB_d\cB_d\tran)I \leq \cB_d\cB_d\tran }
where $\lambda_{\min}(\cdot)$ denotes the smallest eigenvalue of its argument. Therefore, 
\eq{
\|\tsy_{i-1}\|_{I-\mu_w \mu_y \cB_d\cB_d\tran}^2 \leq \big(1-\mu_w \mu_y\lambda_{\min}(\cB_d\cB_d\tran)\big)\|\tsy_{i-1}\|^2.  \label{bound_yy}
}
Finally, since $\ssx_{-1}=0$ and ${\ssx}^o$ are in the range space of $\cL$, the error quantity $\tsx_{i-1}$ always belongs to the range space of $\cL$. This implies that $\|\tsx_{i-1}\|_{\cL^2}^2 \geq \underline{\sigma}^2(\cL) \|\tsx_{i-1}\|^2$ where $\underline{\sigma}^2(\cL)$ denotes the minimum non-zero singular value of $\cL$  -- see \cite[Lemma 1]{alghunaim2019linear}. Therefore,
\eq{
\|\tsx_{i-1}\|_{I-\cL^2}^2 \leq (1-\underline{\sigma}^2(\cL)) \|\tsx_{i-1}\|^2 \label{bound_xx}
}
 Substituting the bounds \eqref{bound_grad_W}, \eqref{bound_Y_Z}, \eqref{bound_yy}, and \eqref{bound_xx} into \eqref{equal_lemma} gives:
\eq{
&\|\tsw_i\|^2_{I-\mu_y \mu_w \cB_d\tran \cB_d}+c_y\|\tsy_{i} \|^2+c_y\|\tsx_{i}\|^2 \nnb
& \leq  \gamma_1 \|\tsw_{i-1}\|_{I-\mu_y \mu_w \cB_d\tran \cB_d}^2+\gamma_2 c_y\|\tsy_{i-1}\|^2 +  \gamma_3 c_y\|\tsx_{i-1}\|^2, \label{last_inq}
}
where $\gamma_2 \define  1-\mu_w \mu_y\lambda_{\min}(\cB_d\cB_d\tran)$ and $\gamma_3 \define  1-\underline{\sigma}^2(\cL)$. Under the step-size conditions given in \eqref{step_size_c}, 
 it holds that $I-\mu_y \mu_w \cB_d\tran \cB_d>0$ and $\gamma=\max\{\gamma_1,\gamma_2,\gamma_3\}<1$. Moreover, we have that $(1-\mu_y \mu_w \sigma_{\max}^2(\cB_d)) \|\tsw_i\|^2 \leq \|\tsw_i\|^2_{I-\mu_y \mu_w \cB_d\tran \cB_d}$. Let
\eq{
C_o= {\|\tsw_{-1}\|_{I-\mu_y \mu_w \cB_d\tran \cB_d}^2+ c_y\|\tsy_{-1}\|^2 +   c_y\|\tsx_{-1}\|^2  \over 1-\mu_y \mu_w \sigma_{\max}^2(\cB_d)}.
} 
Iterating inequality \eqref{last_inq}  yields the result.

\end{document}